\documentclass[12pt]{amsart}
\usepackage{amssymb}
\usepackage{amsmath}
\usepackage{color,epsfig}
\usepackage{graphicx}
\usepackage{mathdots}

\usepackage[pdftex,colorlinks,backref,pagebackref,hypertexnames=false,linkcolor=blue,urlcolor=green,citecolor=red]{hyperref}

\usepackage{bookmark}

\def\F{\mathbb F}

\def\N{\mathbb N}

\def\cB{\mathcal B}
\def\cC{\mathcal C}
\def\cD{\mathcal D}

\def\cI{\mathcal I}
\def\cC{\mathcal C}

\def\cP{\mathcal P}

\newcommand{\im}{{\sf img}}

\newcommand{\Stab}{{\sf Stab}}

\newcommand{\SL}{{\sf SL}}

\newcommand{\GL}{{\sf GL}}

\newcommand{\Aut}{{\sf Aut}}

\newcommand{\la}{\langle}\newcommand{\ra}{\rangle}

\newcommand{\ov}{\overline}
\newcommand{\ds}{\displaystyle}

\newtheorem{theorem}{Theorem}[section]
\newtheorem{lemma}[theorem]{Lemma}

\newtheorem{proposition}[theorem]{Proposition}

\newtheorem{definition}[theorem]{Definition}
\newtheorem{example}[theorem]{Example}
\newtheorem{remark}[theorem]{\sf\em Remark}
\newtheorem{question}[theorem]{\sf\em Question}

\begin{document}
	
	\title[On $\Aut(\cD)$ of $2$-designs $\cD$ constructed from $\F_{2^n}$]{On the full automorphism groups of $2$-designs constructed from finite fields $\F_{2^n}$
	}

\date{\today}

\author{Tung Le and B.~G.~Rodrigues}

\email{T.L.: Lttung96@yahoo.com}
\email{B.G.R.: bernardo.rodrigues@up.ac.za}

\address{T.L: Department of Mathematics and Applied Mathematics, University of Pretoria, Hatfield 0028, Pretoria, South Africa}
\address{B.G.R.: Department of Mathematics and Applied Mathematics, University of Pretoria, Hatfield 0028, Pretoria, South Africa}

\begin{abstract}
{\small 
	In this manuscript, for $q:=2^n$ with $n\geq2$, we study two primitive maximal subgroups of the alternating group ${\sf A}_{q-1}$. These  subgroups are the full automorphism groups of $2$-designs which are constructed from algebraic curves 
	over the finite field  $\F_q$. 
}
\end{abstract}

\keywords{Automorphism, design, finite field, finite group, simple group, symmetric group}

\subjclass[2010]{05B05, 12E05, 20B15, 20B25, 20B35, 20D06}

\maketitle

\section{Introduction}
Combinatorial designs with equal parameters to those of classical designs
have been studied extensively, and in all cases from a geometric perspective. Here we present a perspective on which algebraic curves defined over finite field $\F_q$  provide an interplay between finite algebraic geometry and finite discrete structures, namely combinatorial design theory. In details, we describe the designs defined by parabolas and hyperbolas respectively, and 
study the structure of their full automorphism groups in the context of the symmetric group ${\sf Sym}(\F_q^\times)$.

With $q := 2^n$ where $n\in\N_{\geq2}$, for every subset $S \subseteq \F_q,$  we write $S^\times: = S - \{0\}.$ For each $a \in \F_q^\times$, we define the following parabola and hyperbola \begin{center} $U_a^q :=\{x^2 + ax \mid x \in \F_q\}$ and $O_a^q:=\{x+ax^{-1}\mid x\in\F_q^\times\}.$ \end{center}  We prove the following main result.

\begin{theorem}
	\label{thm:main1}
	For $q:=2^n$ where $n\in\N_{\geq2}$, let $\cD_q^u:=(\F_q^\times, \cB_q^u)$ and $\cD_q^o =(\F_q^\times, \cB_q^o)$ be two incident structures, where $\cB_q^o:=\{{(O_a^q)}^\times: a \in \F_q^\times\}$ and $\cB_q^u:=\{{(U_a^q)}^\times: a \in \F_q^\times\}.$ The following results hold.
	\begin{itemize}
		\item[(i)] 
		$\cD_q^u$ and $\cD_q^o$ are 
		$2$-$(q-1, \frac{q}{2}-1, \frac{q}{4}-1)$ designs. Moreover, $\cD_q^o$ is design-isomorphic to the dual of $\cD_q^u$.  Thus, the two full automorphism groups $\Aut(\cD_q^u)$ and $\Aut(\cD_q^o)$ are isomorphic.
		
		\item[(ii)] $\Aut(\cD_q^u)= \GL_{\F_2}(\F_q)\cong\GL_n(2)$, where $\F_q$ is considered as an $\F_2$-vector space. 
		
		\item[(iii)] Both $\Aut(\cD_q^u)$ and $\Aut(\cD_q^o)$ are primitive maximal subgroups of the alternating group ${\sf Alt}(\F_q^\times)$. 	Furthermore, 
		let $H$ be either $\Aut(\cD_q^u)$ or $\Aut(\cD_q^o)$,
		the stabilizer of a point and the stabilizer of a block in $H$ are both isomorphic to $2^{n-1}{:}\GL_{n-1}(2)$.
		\begin{itemize}
			\item If $n=2$, then these two stabilizers are conjugate in $H$.
			\item If $n\geq3$, then these two stabilizers are not conjugate in $H$.
		\end{itemize}		
	\end{itemize}	
\end{theorem}

Notice that for each $a\in\F_q^\times$, the set $U_a$ is a hyperplane in the $\F_2$-vector space $\F_q$, the design $\cD_q^u$ is 
 a classical design, known as the projective space ${\sf PG}(n{-}1,2)$. However, $\cD_q^o$ is not a classical design as $O_a^q$ is not closed under addition for every $a\in\F_q^\times$.

The proof of Theorem \ref{thm:main1} follows by a series of lemmas and theorems in Section~\ref{sec:designs} and Section \ref{sec:automorphism}.
The paper is organized as follows. In Section~\ref{sec:designs}, we introduce the combinatorial designs constructed from algebraic curves 
(parabolas and hyperbolas) over the finite field $\F_q$ where $q := 2^n$, 
provide their parameters, and show that these designs are isomorphic to the dual to each other. In Section~\ref{sec:automorphism}, we determine the structure and the primitive action of the full automorphism groups $\Aut(\cD_q^u)$  and $\Aut(\cD_q^o)$ on points and blocks,  prove the non-conjugation of the stabilizers, and then prove their maximality in the alternating group ${\sf Alt}(\F_q^\times)$ as mentioned in Theorem \ref{thm:main1} (iii). Finally, we put forward
 two open questions 
for these automorphism groups in Question~\ref{que:ONanScott}.

\section{\bf Terminology}\label{sec-tan}
The notation for designs is as in \cite{ak:book}. An incidence structure $\cD:=(\cP,\cB,\cI)$, with point set $\cP,$ block set $\cB$ and incidence relation $\cI$ is a {\em $t$-$(v,k,\lambda)$ design}, if $|\cP|=v$, every block $B \in \cB$ is incident with precisely $k$ points, and every $t$ distinct points are together incident with precisely $\lambda$ blocks. The {\em dual} of $\cD$ is a design $\cD^* := (\cB, \cP)$, where $\cB$ corresponds to a set of points and $\cP$ to a set of blocks, such that $B \in \cB$ is a point incident with a block $p \in \cP$ if and only if $p$
is incident with $B$ in $\cD$.
The {\em complementary design} is obtained by replacing all blocks by their complements. The design $\cD$ is {\em symmetric} if it has the same number of points and blocks.

For designs $\cD:=(\cP,\cB)$ and $\cD':=(\cP',\cB')$, we say $\cD$ and $\cD'$ are {\em (design) isomorphic} if there is a bijective map $\sigma$ from $\cP$ onto $\cP'$ such that the image $\sigma(B)$ of every block $B\in\cB$ is a block in $\cB'$, and $\sigma$ is called a {\em (design) isomorphism}.
If $\cD=\cD'$, then $\sigma$ is called an {\em automorphism} of $\cD$.
The set of all automorphisms of $\cD$ forms a group under composition, called the {\em (full) automorphism group} of $\cD$, denoted by \Aut($\cD$).

Our notation for groups will be standard, and it is as in \cite{Asc86}.

\medskip

We start by recalling the following property of a finite dimensional $\F_q$-vector space $V$.

\begin{lemma}
	\label{lem:Hyp}
	Let $V$ be an $n$-dimensional vector space over 
	 $\F_q$. Then $V$ has exactly $\frac{q^n-1}{q-1}$ hyperplanes, i.e. $V$ has exactly $\frac{q^n-1}{q-1}$ vector subspaces of dimension $n{-}1$.
\end{lemma}

\section{\bf Designs constructed from algebraic curves}
\label{sec:designs}

First we define blocks.

\begin{definition} (Hyperbolas)
\label{defn:hyperb}
	For every $a\in\F_q^\times$, we define $$O_a^q:=\{x+ax^{-1}:x\in\F_q^\times\}=\{b\in \F_q:x^2+bx+a\text{ splits over }\F_q\}, \text{ and}$$
 $$\cB_q^o:=\{(O_a^q)^\times:a\in\F_q^\times\}\text{ and }\cD_q^o:=(\F_q^\times,\cB_q^o).$$
\end{definition}

 We have $0\in O_a^q$ as there is $u\in\F_q^\times$ such that $u^2=a\Leftrightarrow u+au^{-1}=0$, and $|O_a^q|=\frac{q}{2}$ as $x+ax^{-1}=y+ay^{-1}$ if and only if $(x+y)(a+xy)=0$,  if and only if $\{x,ax^{-1}\}=\{y,ay^{-1}\}$. Thus, $|(O_a^q)^\times|=\frac{q}{2}-1$.

Consider the regular action of $\F_q^\times$ on itself, i.e. the (left) multiplication of $\F_q^\times$ on $\cP$. Thus, $\F_q^\times$ acts transitively on $\cP$. For every $t\in\F_q^\times$, we denote $\nu_t:\F_q\rightarrow\F_q,u\mapsto tu$. We have $$\nu_t(O_a^q)=tO_a^q=\{tx+tax^{-1}:x\in\F_q^\times\}=\{y+t^2ay^{-1}:y\in\F_q^\times\}=O_{t^2a}^q.$$
This implies that the map $\nu_t:\cB_q^o\rightarrow\cB_q^o,(O_a^q)^\times\mapsto (O_{t^2a}^q)^\times$ is well-defined. As the square map is a field automorphism of $\F_q$, this multiplication action of $\F_q^\times$ is transitive on $\cB_q^o$.

\begin{lemma}
	\label{lem:reg_act}
	$\F_q^\times$ acts regularly on $\cB_q^o$. Thus, $|\cB_q^o|=q-1$.
\end{lemma}
\begin{proof}
	Since $\F_q^\times$ acts transitively on $\cB_q^o$ via $tO_a^q=O_{t^2a}^q$ for all $t\in\F_q^\times$, it suffices to show that the stabilizer $\Stab_{\F_q^\times}(O_a^q)=1$ for some $a\in\F_q^\times$. 	
	Since $(\F_q^\times,\cdot)$ is cyclic, let $x\in\F_q^\times$ such that $\F_q^\times=\la x\ra=\{1,x,\ldots,x^{q-2}\}$.
	
	For a contradiction, we assume that $\Stab_{\F_q^\times}(O_a^q)=\la s\ra\neq 1$ where $s=x^m$ with the smallest positive $m\leq q{-}2$. By the transitivity of $\F_q^\times$ on $\cB_q^o$, we have $|\cB_q^o|=(q-1)/r$ where $r:=|\Stab_{\F_q^\times}((O_a^q)^\times)|$. Thus, 	$mr=q{-}1$ and $r\,|\,(q{-}1)$.
	
	Write $(O_a^q)^\times=\{a_1,a_2,\ldots,a_k\}$ where $k=\frac{q}{2}-1$.	Since $s(O_a^q)^\times=\{sa_1,\ldots,sa_k\}=(O_a^q)^\times$, the stabilizer $\la s\ra$ acts on $(O_a^q)^\times$. Here, the orbit of $a_i$ is $a_i\la s\ra$ of size $r$. As $(O_a^q)^\times$ is a union of disjoint orbits whose size is $r$, we have $r\,|\,(\frac{q}{2}-1)$.
	
	Notice that $2(\frac{q}{2}-1)=q{-}2$ and $\gcd(q{-}2,q{-}1)=1$. Since $r\,|\,(q{-}1)$ and $r\,|\,(\frac{q}{2}-1)$, this forces $r=1$. This contradicts the fact that the order of $\Stab_{\F_q^\times}((O_a^q)^\times)$ is nontrivial.
	
	 Therefore, $\Stab_{\F_q^\times}((O_a^q)^\times)=1$	 and $\F_q^\times$ acts regularly on $\cB_q^o$.	
\end{proof}


\smallskip

\begin{definition} \label{defn:parab} (Parabolas)
	For each $a\in\F_q^\times$, we define the map $\lambda_a:\F_q\rightarrow\F_q,x\mapsto x^2+ax$. Denote 
	$\im(\lambda_a),$ the image of $\lambda_a$.
	For each $a\in\F_q^\times$, we define
	\begin{center}
		$U_a^q:=\im(\lambda_a)$, 
	 $\cB_q^u:=\{(U_a^q)^\times:a\in\F_q^\times\}$, and $\cD_q^u:=(\F_q^\times,\cB_q^u)$.
	 \end{center}
\end{definition}

 Since ${\sf char}(\F_q)=2$, the map $\lambda_a$ is an additive group homomorphism whose kernel is $\{0,a\}$. Thus, the image $\im(\lambda_a)$ of $\lambda_a$ is an $\F_2$-hyperplane of $\F_q$ of size $\frac{q}{2}$. Further, we have $$U_a^q=\{x^2+ax:x\in\F_q\}=\{b\in\F_q:x^2+ax+b\textit{ splits over }\F_q\,\}.$$

For each $t,a\in\F_q^\times$, we have $$\nu_{t^2}(U_a^q)=t^2U_a^q=\{t^2x^2+t^2ax:x\in\F_q\}=\{y^2+tay:y\in\F_q\}=U_{ta}^q.$$
This implies that the map $\nu_{t^2}:\cB_q^u\rightarrow\cB_q^u,(U_a^q)^\times\mapsto(U_{ta}^q)^\times$ is well-defined. With the square field automorphism, the multiplication action of $\F_q^\times$ is transitive on $\cB_q^u$. With a similar argument in the proof of Lemma \ref{lem:reg_act}, we can show that the multiplicative group $\F_q^\times$ acts regularly on $\cB_q^u$, and $|\cB_q^u|=q-1$.

\begin{remark}\label{rem:Hyper}
	{\rm $(i)$ Note that there is $a\in\F_q^\times$ such that $\ker(\lambda_a)=\{0,a\}\subseteq \im(\lambda_a)$. For example, with $q=2^{2r}$, the polynomial $x^2+x+1$ splits over $\F_q$, and thus $1\in \im(\lambda_1)\cap \ker(\lambda_1)$.
	
	$(ii)$ By Lemma \ref{lem:Hyp}, the set $\{U_a^q:a\in\F_q^\times\}$ is the set of all $\F_2$-hyperplanes of $\F_q$, i.e. for each $\F_2$-hyperplane $H$ of $\F_q$, there is a unique $a\in\F_q^\times$ such that $H=U_a^q$.}
\end{remark}

\smallskip

For every $a,v\in\F_q^\times$, we have $$v\in(O_a^q)^\times\Leftrightarrow \exists x\in\F_q^\times,\,v=x+ax^{-1}\Leftrightarrow \exists x\in\F_q^\times,\,x^2+vx+a=0\Leftrightarrow a\in (U_v^q)^\times.\ \ \ (*)$$

Now we will show that both geometric structures $\cD_q^o=(\F_q^\times, \cB_q^o)$ and $\cD_q^u=(\F_q^\times,\cB_q^u)$ are $2$-designs. Furthermore, one is design-isomorphic to the dual of the other.

\begin{theorem}
	\label{thm:1design}
	The following results hold.
	\begin{itemize}
		\item[(i)] For every $v\in\F_q^\times$, there are exactly $\frac{q}{2}-1$ elements $a\in \F_q^\times$ such that $v\in O_a^q$.\\ Thus, $\cD_q^o=(\F_q^\times,\cB_q^o)$ is a $1$-$(q{-}1,\frac{q}{2}-1,\frac{q}{2}-1)$ design.
		\item[(ii)] For every $a\in\F_q^\times$, there are exactly $\frac{q}{2}-1$ elements $v\in \F_q^\times$ such that $a\in U_v^q$.\\ Thus, $\cD_q^u=(\F_q^\times,\cB_q^u)$ is a $1$-$(q{-}1,\frac{q}{2}-1,\frac{q}{2}-1)$ design.
		\item[(iii)] 	The design $\cD_q^o=(\F_q^\times,\cB_q^o)$ is isomorphic to the dual of the design $\cD_q^u=(\F_q^\times,\cB_q^u)$.
		\item[(iv)] For all $v\neq u\in \F_q^\times$, there are exactly  $\frac{q}{4}-1$ elements $a\in\F_q^\times$ such that $u,v\in O_a^q$.\\ Thus, $\cD_q^o=(\F_q^\times,\cB_q^o)$ is a $2$-$(q{-}1,\frac{q}{2}-1,\frac{q}{4}-1)$ design.
	\end{itemize}
	
\end{theorem}

\begin{proof}
	(i) For each $v\in\F_q^\times$, there are exactly $|(U_v^q)^\times|=\frac{q}{2}-1$ elements $a\in\F_q^\times$ such that $v\in (O_a^q)^\times$.	
	By Lemma \ref{lem:reg_act}, we have $|\cB|=q{-}1$, and $(O_a^q)^\times\neq (O_b^q)^\times$ for all $a\neq b\in\F_q^\times$.
	
	This confirms that each point $v\in\F_q^\times$ is incident with exactly $\frac{q}{2}-1$ blocks 
	 $(O_a^q)^\times$ where $a\in(U_v^q)^\times$. Therefore, $\cD_q^o$ is a $1$-$(q{-}1,\frac{q}{2}-1,\frac{q}{2}-1)$ design.
	
	\smallskip
	
	(ii)  By $(*)$, for each $a\in\F_q^\times$, there are exactly $|(O_a^q)^\times|=\frac{q}{2}-1$ elements $v\in\F_q^\times$ such that $a\in (U_v^q)^\times$. Arguing similarly as in part (i) and applying Lemma \ref{lem:reg_act}, it follows that $\cD_q^u$ is a $1$-$(q{-}1,\frac{q}{2}-1,\frac{q}{2}-1)$ design.
	
	\smallskip
	
	(iii) By definition, the dual of $\cD_q^o$ is $(\cB_q^o,\F_q^\times)$ where a block $v\in\F_q^\times$ is incident with all points $(O_a^q)^\times\in\cB_q^u$ satisfying $v\in (O_a^q)^\times$. In this dual, by (*), a point $(O_a^q)^\times$ is incident with a block $v$ if and only if $a\in (U_v^q)^\times$.
	
	We define the bijection  $\gamma:\cB_q^o\rightarrow\F_q^\times,(O_a^q)^\times\mapsto a$. 
	The map $\gamma$ is well-defined by the regular action of $\F_q^\times$ on $\cB_q^o$.
	By $(*)$, the incidence relations of designs $(\cB_q^u,\F_q^\times)$  and $(\F_q^\times,\cB_q^u)$ induce the bijective map $\bar\gamma:\F_q^\times\rightarrow\cB_q^u,v\mapsto (U_v^q)^\times$, i.e. the image of every bock in $\cD_q^o$ is a block in $\cD_q^u$. Therefore, the bijection $\gamma:(\cB_q^o,\F_q^\times)\rightarrow(\F_q^\times,\cB_q^u)$ is a design isomorphism.

	\smallskip

	(iv) For $x\neq y\in\F_q^\times$, a direct application of $(*)$ shows that $x,y\in O_a^q$ for some $a\in\F_q^\times$ if and only if $a\in (U_x^q)^\times\cap (U_y^q)^\times$.
	
	By Remark \ref{rem:Hyper} $(ii)$,
	if $x\neq y\in\F_q^\times$, then $U_x^q\neq U_y^q$. Since both $U_x^q$ and $U_y^q$ are distinct $\F_2$-hyperplanes of $\F_q$, from $|\F_q|=|U_x^q+U_y^q|=|U_x^q||U_y^q|/|U_x^q\cap U_y^q|$, we have  $|U_x^q\cap U_y^q|=q/4$.
	Thus, $|(U_x^q)^\times\cap (U_y^q)^\times|=\frac{q}{4}-1$, from which we deduce that there are exactly $(\frac{q}{4}-1)$ elements $a\in\F_q^\times$ such that $x,y\in (O_a^q)^\times$. This shows that $\cD_q^o$ is a $2$-$(q{-}1,\frac{q}{2}{-}1,\frac{q}{4}{-}1)$ design.	
\end{proof}

\vspace*{1mm}

\begin{remark}
	{\rm Fix $a\in\F_q^\times$ and write $(O_a^q)^\times=\{a_1,\ldots,a_k\}$ where $k=\frac{q}{2}-1$. For every $u\in\F_q^\times$, we have $u\in ua_i^{-1}O_a^q$ for all $1\leq i\leq k$. 	By Lemma \ref{lem:reg_act}, all these blocks are distinct since  $\F_q^\times$ acts regularly on $\cB_q^o$. 
	This offers an alternative proof of part (i) of Theorem \ref{thm:1design}, and an equivalent argument for Theorem \ref{thm:1design} (ii).}
\end{remark}

\vspace*{1mm}

\begin{remark} \label{rem:T_q}
	{\rm Since the multiplicative group $T_q:=\{\nu_t:t\in\F_q^\times\}\cong(\F_q^\times,\cdot)$ acts transitively on $\F_q^\times$, $\cB_q^o$, and $\cB_q^u$ respectively, it is a transitive subgroup of 
	$\Aut(\cD_q^o)$ and $\Aut(\cD_q^u)$. As subgroups of the symmetric group ${\sf Sym}(\F_q^\times)$, we have $T_q\leq \Aut(\cD_q^o)\cap\Aut(\cD_q^u)$.}
\end{remark}


\vspace*{1mm}

Next, we examine the complementary designs of these designs to uncover their properties and those of the polynomials $x^2+ax+b$ in $\F_q[x]$.

\begin{definition}
	\label{defn:cdesign}
	For every $a\in\F_q^\times$, we define $$\ov{O_a^q}:=\F_q{-}O_a^q=\{b\in\F_q:t^2+bt+a\text{ is irreducible in }\F_q[t]\}.$$
	We define $$\cB_q^{oc}:=\{\ov{O_a^q}:a\in\F_q^\times\}, \text{ and }\cD_q^{oc}:=(\F_q^\times,\cB_q^{oc}) .$$
\end{definition}

\noindent
Here, $\cD_q^{oc}$  
denotes the complementary design of $\cD_q^o$.

\begin{lemma}
	\label{lem:cdesign}
	For all $a\neq b\in\F_q^\times$, $|\ov{O_a^q}\cap\ov{O_b^q}|=\frac{q}{4}$, and $\cD_q^{oc}$ is a $2$-$(q{-}1,\frac{q}{2},\frac{q}{4})$ design.
\end{lemma}
\begin{proof}
	Since $|O_a^q|=\frac{q}{2}$, we have $|\ov{O_a^q}|=\frac{q}{2}$ for every $a\in\F_q^\times$.
	
	For every $x\in\F_q^\times$, there are exactly $(\frac{q}{2}{-}1)$ blocks $(O_a^q)^\times$ containing $x$. Thus, there are exactly $(\frac{q}{2}{-}1)$ blocks $\ov{O_a^q}$ not containing $x$. So there are exactly $\frac{q}{2}$ blocks $\ov{O_b^q}$ containing $x$. This shows that $\cD_q^{oc}=(\F_q^\times,\cB_q^{oc}) $ is a $1$-$(q{-}1,\frac{q}{2},\frac{q}{2})$ design.
	
	For $a\neq b\in\F_q^\times$, we have $\ov{O_a^q}\cap\ov{O_b^q}=\F_q{-}(O_a^q\cup O_b^q)$. Since $|O_a^q|=|O_b^q|=\frac{q}{2}$ and $|O_a^q\cap O_b^q|=\frac{q}{4}$ (including $0$), we have $|O_a^q\cup O_b^q|=\frac{3q}{4}$. Thus, $|\ov{O_a^q}\cap\ov{O_b^q}|=\frac{q}{4}$.

	For $a,x,y\in \F_q^\times$  where $x\neq y$, we have $x,y\in \ov{O_a^q}$  if and only if $x,y\not\in (O_a^q)^\times$. Using the same counting method, one can compute that there are exactly $(\frac{q}{2}{-}1)$ blocks $(O_a^q)^\times$ containing $x$, and exactly $(\frac{q}{2}{-}1)$ blocks $(O_b^q)^\times$ containing $y$, respectively. Furthermore, 
	 there are exactly $(\frac{q}{4}{-}1)$ blocks $(O_a^q)^\times$ containing both $x$ and $y$. Thus, there are exactly $\frac{3q}{4}{-}1$ blocks $(O_a^q)^\times$ containing either $x$ or $y$. So the number of blocks $(O_a^q)^\times$ not containing both $x$ and $y$ is $\frac{q}{4}$. This shows that $\cD_q^{oc}$ is a $2$-$(q{-}1,\frac{q}{2},\frac{q}{4})$ design.
\end{proof}

\smallskip

\begin{lemma}
	\label{lem:dualDoc}
	The dual of $\cD_q^{oc}$ is isomorphic to the design $(\F_q^\times,\{\ov{U_a^q}:a\in\F_q^\times\})$ where $\ov{U_a^q}:=\F_q{-}U_a^q=\{b\in\F_q:t^2+at+b\text{ is irreducible over }\F_q\,\}$. 
\end{lemma}


\begin{proof}
By definition, the dual of $\cD_q^{oc}$ is $(\cB_q^{oc},\F_q^\times)$. Here, one can consider each block $b\in\F_q^\times$ as $b_d:=\{\ov{O_a^q}\in\cB_q^{oc}:a\in\F_q^\times \text{ and }b\in \ov{O_a^q}\}$.

Since $b\in\ov{O_a^q}$ if and only if $t^2+bt+a$ is irreducible over $\F_q$, one can identify the dual of $\cD_q^{oc}$ with the design for which 
the point set is $\F_q^\times$ and the block set 
is $\{B_b:b\in\F_q^\times\}$ where $B_b:=\{a\in\F_q^\times: t^2+bt+a\text{ is irreducible over }\F_q\}=\F_q{-}U_b^q=\ov{U_b^q}$ as claimed. 
\end{proof}


\section{\bf The full automorphism group of the design $\cD_q^u$}
\label{sec:automorphism}

For an $\F_2$-vector space $V$, we write $\GL_{\F_2}(V)$ the set of all 
$\F_2$-linear automorphisms of $V$. We write $\GL_n(q)$ the set of all invertible $n{\times}n$-matrices over $\F_q$. In this section, we prove the following result.

\begin{theorem}
	\label{them:AutD}
	Let $q:=2^n$ and $G:=\Aut(\cD_q^u)$ where $\cD_q^u:=(\F_q^\times,\cB_q^u)$ is the design given in Definition $\ref{defn:parab}$. The following results hold.
	\begin{itemize}
		\item[(i)]  $G=\GL_{\F_2}(\F_q)\cong\GL_n(2)$ where $\F_q$ is considered as an $\F_2$-vector space.
		\item[(ii)]  $G$ acts primitively on points and blocks respectively, i.e. $G$ acts transitively on points and blocks, and both stabilizers $\Stab_G(1)$, $\Stab_G(U_1^q)$ are maximal in $G$.
		\item [(iii)] Both the stabilizers $\Stab_G(1)$ and $\Stab_G(U_1^q)$ are  isomorphic to $2^{n-1}{:}\GL_{n-1}(2)$.\\ Whenever $n\geq3$, these two stabilizers are not conjugate in $G$.
		\item [(iv)] 
		$G$ is a primitive maximal subgroup of the alternating group ${\sf Alt}(\F_q^\times)$.
	\end{itemize}
\end{theorem}

We will prove this theorem in Lemmas \ref{lem:2PIntBlks}, \ref{lem:AutinGL}, \ref{lem:Stab}, and Proposition \ref{prop:subgpAn}.

\smallskip

\begin{remark}
	\label{rem:Du-Do}
{\rm	By Theorem \ref{thm:1design} (iii), almost all the results for $\Aut(\cD_q^u)$ in Theorem \ref{them:AutD} hold for $\Aut(\cD_q^o)$, except (i), i.e. $\Aut(\cD_q^o)\cong \GL_n(2)$ but $\Aut(\cD_q^o)\neq \GL_{\F_2}(\F_q)$.}
\end{remark}

\smallskip


\begin{lemma}
	\label{lem:2PIntBlks}
	In the design	$\cD_q^u=(\F_q^\times,\cB_q^u)$, for every $a\neq b\in\F_q^\times$, the intersection of all blocks containing $a,b$ equals $\{a,b,a+b\}$.
\end{lemma}

\begin{proof} For $a\neq b\in \F_q^\times$, we first identify all blocks in $\cB_q^u$ such that $a,b\in (U_u^q)^\times$.
	
	As $\la a,b\ra$ is a $2$-dimensional subspace of $\F_q$, by Lemma \ref{lem:Hyp}, it suffices to find all $\F_2$-hyperplanes of $\F_q$ containing $a,b$. These correspond to $\F_2$-hyperplanes of the quotient vector space $\F_q/\la a,b\ra$. Here, the intersection of these $\frac{q}{4}-1$ $\F_2$-hyperplanes of $\F_q/\la a,b\ra$ is the trivial subspace $\la a,b\ra$. This shows that the intersection of all blocks  containing $a,b$ in $\cD_q^u$ equals $\{a,b,a+b\}$.
\end{proof}	

\smallskip

\begin{remark}
	{\rm $(i)$ The above argument 
	gives an alternative 
	 proof for $\cD_q^u$ to be a $2$-design.
	
	$(ii)$ This also proves that for $a\neq b\in\F_q^\times$, there are exactly $\frac{q}{4}{-}1$ elements $u\in\F_q^\times$ such that $x^2+ux+a$ and $x^2+ux+b$ split over $\F_q$.}
\end{remark}

\smallskip

\begin{lemma}
	\label{lem:AutinGL}
	If $\F_q$ is considered as an $\F_2$-vector space, then $\Aut(\cD_q^u)= \GL_{\F_2}(\F_q)\cong\GL_n(2)$.
\end{lemma}

\begin{proof}
	\underline{$\Aut(\cD_q^u)\subseteq \GL_{\F_2}(\F_q)$:}
	For $\tau\in\Aut(\cD_q^u)$, let $\tau(0):=0$. As $\tau$ is a bijection of $\F_q$, it suffices to show that $\tau$ is an $\F_2$-linear transformation, i.e. for $u,v\in\F_q$, $\tau(u+v)=\tau(u)+\tau(v)$.
	
	It is obvious if $u$ or $v$ equals $0$, and if $u=v$. We assume that $u\neq v\in\F_q^\times$. For $x,y\in\F_q^\times$, set $B_{x,y}:=\{(U_a^q)^\times\in\cB_q^u:x,y\in (U_a^q)^\times\}$. By Lemma \ref{lem:2PIntBlks} for $\{u,v\}$ and $\{\tau(u),\tau(v)\}$, we have $$\ds S_{u,v}:=\bigcap_{U\in B_{u,v}}U=\{u,v,u+v\} \text{ and } S_{\tau(u),\tau(v)}:=\bigcap_{U\in B_{\tau(u),\tau(v)}}U=\{\tau(u),\tau(v),\tau(u+v)\}.$$
	Since $\tau\in\Aut(\cD_q^u)$, we have $\tau:B_{u,v}\rightarrow B_{\tau(u),\tau(v)}$ is bijective, and so is $\tau:S_{u,v}\rightarrow S_{\tau(u),\tau(v)}$. As $u,v,u+v$ are the only elements in $S_{u,v}$ and so are $\tau(u),\tau(v),\tau(u)+\tau(v)$ in $S_{\tau(u),\tau(v)}$, it must follow that $\tau(u+v)=\tau(u)+\tau(v)$. Thus, $\tau\in \GL_{\F_2}(\F_q)$.
	
	\smallskip
	
	\underline{$\GL_{\F_2}(\F_q)\subseteq\Aut(\cD_q^u)$:} For each 
	 $\tau\in\GL_{\F_2}(\F_q)$, by Remark \ref{rem:Hyper} $(ii)$, we have $\tau((U_a^q)^\times)\in\cB_q^u$ 
	  for all $(U_a^q)^\times\in\cB_q^u$ since the image of every $\F_2$-hyperplane is also an $\F_2$-hyperplane in $\F_q$. This shows that $\tau|_{\F_q^\times}\in\Aut(\cD_q^u)$, and thus the claim holds.	
\end{proof}

\smallskip

\begin{example}
	\label{ex:Fano}
	{\rm When $n=3$, we have $\cD_8^u$  and $\cD_8^o$ are $2$-$(7,3,1)$ designs, both of which are isomorphic to the Fano plane, and $\Aut(\cD_8^u)=\GL_{\F_2}(\F_8)\cong \GL_3(2)\cong\SL_2(7)$.}
\end{example}

\smallskip

\begin{remark}
	\label{rem:AutO&U}
	 {\rm The fact that the dual of $\cD_q^u$ is isomorphic to $\cD_q^o$ does not imply $\Aut(\cD_q^o){=}\Aut(\cD_q^u)$. Here we can only obtain that $\Aut(\cD_q^o)$ is isomorphic to $\GL_{\F_2}(\F_q)$. One can use Magma \cite{magma} to confirm for $n=3$ that $\Aut(\cD_8^o)\neq \Aut(\cD_8^u)$; in detail, $|\Aut(\cD_8^o)\cap  \Aut(\cD_8^u)| =21$.}
\end{remark}

\smallskip

\begin{remark}
	\label{rem:SingerTorus}
	{\rm Recall the multiplicative subgroup $T_q:=\{\nu_a:a\in\F_q^\times\}\leq\Aut(\cD_q^o)\cap\Aut(\cD_q^u)$ in Remark \ref{rem:T_q}. Since $\Aut(\cD_q^u)=\GL_{\F_2}(\F_q)\cong\GL_n(2)$, the subgroup $T_q\cong(\F_q^\times,\cdot)$ is a maximal nonsplit torus of order $q-1$, also known as a Singer torus of $\GL_n(2)$.}
\end{remark}

\smallskip

\begin{remark}
	\label{rem:Frobenius}
	{\rm Consider the Frobenius map $\theta:\F_q\rightarrow\F_q,\,x\mapsto x^2$. For every $a\in\F_q^\times$, we have $\theta(U_a^q)=U_{a^2}^q$ and $\theta(O_a^q)=O_{a^2}^q$. So the cyclic group $\la\theta\ra$ is a subgroup of $\Aut(\cD_q^u)$ and $\Aut(\cD_q^o)$. Notice that $\theta\in\Stab_{\Aut(\cD_q^u)}((U_1^q)^\times)\cap\Stab_{\Aut(\cD_q^o)}((O_1^q)^\times)$.

	As $q=2^n$, the order of $\theta$ equals $n$. So the normalizer subgroup $N_{\GL_{\F_2}(\F_q)}(T_q)=\la \theta,T_q\ra$ and the Weyl group $W(T_q):=N_{\GL_{\F_2}(\F_q)}(T_q)/T_q\cong C_n$.}

\end{remark}


\smallskip


\begin{lemma}
	\label{lem:Stab}
	With $q:=2^n$ where $n\geq 2$, let $G:=\Aut(\cD_q^u)$, $a\in\F_q^\times$ and $b\in\cB_q^u$.
	The following results hold.
	\begin{itemize}
		\item [(i)] $G$ acts primitively on both points and blocks 
		where the stabilizers $\Stab_G(a)$ and $\Stab_G(b)$ are isomorphic to $2^{n-1}{:}\GL_{n-1}(2)$
		\item[(ii)] Under the action of $\Stab_G(b)$ on the point set $\F_q^\times$, there are exactly $2$ orbits whose sizes are $2^{n-1}-1$ and $2^{n-1}$.
		\item[(iii)] If $n\geq 3$, then the stabilizers $\Stab_G(a)$ and $\Stab_G(b)$ are not conjugate in $G$. 		
 	\end{itemize}
\end{lemma}

\begin{proof}
For parts (i)-(ii)	we apply Lemma \ref{lem:AutinGL} with $G:=\Aut(\cD_q^u)=\GL_{\F_2}(\F_q)$. 	
	We consider $V:=\F_q$ as an $\F_2$-vector space of dimension $n$ and choose an $\F_2$-basis $\cC$ for $V$ such that the coordinate vector $[a]_\cC=[0,0,\ldots,0,1]^T$, a column vector. For every $\sigma\in  G$, with respect to
	 the basis $\cC$, the coordinate vector $[\sigma(a)]_\cC=[\sigma]_\cC\,[a]_\cC$ where 
	 $[\sigma]_\cC:=[\sigma]_{\cC,\cC}\in\GL_n(2)$ is the matrix form of $\sigma$ with respect to
	 the bases from $\cC$ to $\cC$. Thus, the matrix subgroup $\left\{[\sigma]_{\cC}\mid\sigma\in\Stab_G(a)\right\}$ is  $\left\{\begin{bmatrix}
		A&{\bf 0}\\ * &1
	\end{bmatrix}\mid A\in\GL_{n-1}(2)\right\}$, which is a maximal parabolic subgroup of $G$ and isomorphic to $2^{n-1}{:}\GL_{n-1}(2)$. This also confirms that $G$ acts primitively on the points $\F_q^\times$.
		
	In terms of coordinate vectors with respect to $\cC$: The orbit of $[a]_\cC$ is $\{[x,1]^T:x\in\F_2^{n-1}\}$ of size $2^{n-1}$. The orbit 
 with representative $[1,0,\ldots,0]^T$ is the set $\{[x,0]^T: 0\neq x\in \F_2^{n-1}\}$ of size $2^{n-1}-1$. Thus, $\Stab_G(a)$ acts on the points into these two orbits.
	
	Since $\cD_q^o$ and $\cD_q^u$ are dual to each other, using the same argument for the primitive action of $\Aut(\cD_q^o)$ on the points, we obtain the primitive action of $\Aut(\cD_q^u)$ on the blocks and $\Stab_G(b)$ is isomorphic to $2^{n-1}{:}\GL_{n-1}(2)$ for every block $b\in\cB_q^u$.

	\smallskip
	
	(iii) 
	For every $g\in G$, we have $\Stab_G(a)^g=\Stab_G(a^g)$, i.e. $\{a^g\}$ is an orbit of $\Stab_G(a)^g$ under the action of $\Stab_G(a)^g$ on the points.	
	 As $n\geq 3$, $\Stab_G(b)$ and $\Stab_G(a)^g$ have different orbit sizes on the point set, and this implies that these two groups are not conjugate in $G$.
\end{proof}

\smallskip

\begin{remark}
	\label{rem:n=2}
	{\rm If $q=2^2$, then $G:=\Aut(\cD_q^u)=\GL_{\F_2}(\F_4)\cong\GL_2(2)\cong S_3$ and $\Stab_G(1)$ and $\Stab_G(U_1^q)$ are of order $2$. Thus, these two stabilizers are conjugate in $G$.}
\end{remark}

\smallskip

\begin{remark}
	\label{rem:DoOrbits}
	{\rm
	By Theorem \ref{thm:1design} (iii) and Lemma \ref{lem:Stab}, under the action of $\Stab_{\Aut(\cD_q^o)}(1)$ on the block set $\cB_q^o$, there are exactly 2 orbits of sizes $2^{n-1}-1$ and $2^{n-1}$.}
\end{remark}

\smallskip
\begin{proposition}
	\label{prop:subgpAn}
	For $q:=2^n$ where $n\geq 3$, $\Aut(\cD_q^u)$ and $\Aut(\cD_q^o)$ are primitive maximal subgroups of the alternating group ${\sf Alt}(\F_q^\times)$.
\end{proposition}

\begin{proof}
	Let $G$ be either $\Aut(\cD_q^u)$ or $\Aut(\cD_q^o)$. It is clear that $G$ is a subgroup of the symmetric group ${\sf Sym}(\F_q^\times)$ due to the action of $G$ on the point set $\F_q^\times$. 	Now we show that $G$ is a subgroup of the alternating group ${\sf Alt}(\F_q^\times)$.

	By Lemma \ref{lem:AutinGL}, $G\cong\GL_n(2)=\SL_n(2)$, $G$ is simple since $n\geq 3$. By the simplicity of $G$, we have $|G:G\cap {\sf Alt}(\F_q^\times)|>2$ if $G$ is not a subgroup of ${\sf Alt}(\F_q^\times)$. 
	From $|{\sf Sym}(\F_q^\times)|=2|{\sf Alt}(\F_q^\times)|$ and $|G{\sf Alt}(\F_q^\times)|=|G|{\cdot}|{\sf Alt}(\F_q^\times)|/|G\cap {\sf Alt}(\F_q^\times)|$
	, we have $G$ is a subgroup of ${\sf Alt}(\F_q^\times)$.
	
	By Lemma \ref{lem:Stab}, $G$ is a primitive subgroup of ${\sf Alt}(\F_q^\times)$. By \cite[Main Theorem]{LPS}, $G$ is almost simple of type (f), thus, $G$ is maximal in ${\sf Alt}(\F_q^\times)$ as claimed.
\end{proof}

\smallskip

\begin{remark}
	When $q=2^2$, we have $\Aut(\cD_q^u)$ and $\Aut(\cD_q^o)$ are isomorphic to $\GL_2(2)\cong S_3$, and ${\sf Sym}(\F_4^\times)\cong S_3$. Thus, $\Aut(\cD_q^u)=\Aut(\cD_q^o)={\sf Sym}(\F_4^\times)$.
\end{remark}

\smallskip
With a few tests for small $q:=2^n$ where $n\geq3$, we formulate two questions to study further the structures of these two subgroups in the context of the symmetric group ${\sf Sym}(\F_q^\times)$.

\begin{question}
	\label{que:ONanScott}	
{\rm  
	For $q:=2^n$ where $n\geq3$, do the following property hold?
	\begin{itemize}
		\item [(i)] $\Aut(\cD_q^u)\cap\Aut(\cD_q^o)=\la \theta,T_q\ra$, as mentioned in Remark \ref{rem:Frobenius}.		
		\item [(ii)]
		$\Aut(\cD_q^u)$ and $\Aut(\cD_q^o)$ are not conjugate in ${\sf Alt}(\F_q^\times)$, but they are conjugate in ${\sf Sym}(\F_q^\times)$.	
	\end{itemize}
}	
\end{question}


\section{\bf Acknowledgments} \label{acknowledge}
The first and the second author acknowledge research support by the National Research Foundation of South Africa (Grant Number 151764 and Grant Number CPRR23041894647 respectively). The first author would like to thank the Isaac Newton Institute for the support during the visit in 2022.

\medspace
\end{document}